\documentclass[final,12pt]{article}

\usepackage{amssymb,amsthm}
\usepackage{amsmath}
\newtheorem{thm}{Theorem}
\newtheorem{lem}{Lemma}
\newtheorem{corolary} {Corollary}
\theoremstyle{definition}
\newtheorem{remark}{Remark}
\newtheorem{defn}{Definition}
\newtheorem*{acknowledgments}{Acknowledgments}

\def\H{\mathcal H}

\begin{document}

\begin{titlepage}
\title{Spectra of Random Operators with absolutely continuous Integrated Density of States
\footnotetext{%
Mathematics Subject Classification(2010):
47A10, 
47B80
81Q10,
35P05.
}
\footnotetext{%
Keywords:
Density of states, Random Operators, Spectrum.
}}

\author{
\textbf{Rafael del Rio}
\\
\small Departamento de Fisica Matematica\\[-1.6mm]
\small Instituto de Investigaciones en Matem\'aticas Aplicadas y en Sistemas\\[-1.6mm]
\small Universidad Nacional Aut\'onoma de M\'exico\\[-1.6mm]
\small C.P. 04510, M\'exico D.F.\\[-1.6mm]
\small\texttt{delrio@iimas.unam.mx} or \small\texttt{delriomagia@gmail.com}
\\[4mm]
}
\date{}
\maketitle
\vspace{-4mm}
{\it Dedicated to Prof. Fritz Gesztesy on the occasion of his 60th Birthday}
\begin{center}
\begin{minipage}{5in}
  \centerline{{\bf Abstract}}
\bigskip
The structure of the spectrum of random operators is studied. It is shown that
 if the density of states measure of some subsets of the spectrum  is zero, then   
 these subsets are  empty.
In particular follows that  absolute continuity of the IDS
implies  singular spectra of ergodic operators  is either empty or of positive measure.
Our results apply to   Anderson and  alloy type models, perturbed  Landau Hamiltonians,
 almost periodic potentials and  models which are not ergodic.

\end{minipage}
\end{center}

\thispagestyle{empty}
\end{titlepage}
\textwidth=125 mm 
\textheight=195 mm.

\section{Introduction}
 Here we  study some aspects about  the structure of the almost sure spectrum of random operators.
 This was inspired by  
 Barbieri et al. \cite{BMS} where it was shown for some ergodic operators
of Anderson type, that  their almost sure singular continuous spectrum $\Sigma _{sc}$  
 satisfies  either $|\Sigma _{sc}\cap J|>0 $  or $\Sigma _{sc}\cap J=\emptyset$,
where $|\cdot|$ denotes the Lebesgue measure and  $J\subset \mathbb R$ is any interval.
 This is not true for every ergodic operator. The Fibonacci model for example, see \cite {Su}, has
 nonempty  almost sure singular continuous spectrum with zero Lebesgue measure. So, for which 
 models  and for which kinds of spectra  the result of Barbieri et al. holds?.In this work we show that an answer 
 can be given through the  integrated density of states, IDS for short, of the corresponding operators.
  
   In \cite{BMS} nothing is mentioned about IDS and the method used  rely on Howland's theory on relative finite perturbations.
   If the IDS were absolutely continuous the result in \cite{BMS} will follow immediately from Corollary \ref{cor}
   below. The so called Wegner estimates  imply regularity of the IDS and often its absolute continuity. There has been a lot of effort spent on proving  these  estimates, particularly because they provide  a key step  in some methods for  proving localization \cite{Stoll}. 
I have not found in the literature  a proof of the absolute continuity of the IDS with the exact conditions
   on the model given in \cite{BMS}, but there are  results on absolute continuity of the IDS for closely related models.
   In \cite{CH94} Corollary 4.6, the authors prove absolute continuity  for the IDS with conditions very similar to the ones in \cite{BMS}.
    In this case the nonexistence of almost sure spectra of zero measure, in particular singular continuous spectrum, can be proved with techniques    
that depend on the behavior of the IDS as we show in what follows.


     Knowledge of the IDS can give us then information about the measure of almost sure spectra. 
       Our main theorem \ref {main} says that if the density of states measure of a particular  spectrum in an interval is zero, then   
 this spectrum is empty inside that interval. This holds for $\mathbb P$-positive spectra (see definition \ref {as} below),
in particular for almost sure spectra, which could be singular continuous, pure point etc..
    Using the absolute continuity of the IDS that follows from Wegner estimates proved by several  authors for different models,
 we  shall then  see for a variety of situations, including  Anderson and  alloy type models, perturbed  Landau Hamiltonians,
 almost periodic potentials and even models which are not ergodic, that any almost sure spectra  has positive measure, if it is not empty. 
 In section \ref{pre} we present basic  definitions about random operators which are required and mention some important theorems about the 
existence of almost sure spectra for ergodic operators. Here the integrated density of states is introduced.   In section \ref {mai}, we present
the main results. These will allows us to use the absolute continuity of the IDS to prove the mentioned statement about the measure of the almost 
sure spectra. We finish by giving some explicit examples where our results can be applied.

\section {Preliminaries}\label {pre}
Let $(\Omega ,\mathcal M, \mathbb P)$ be a complete probability space .
By $\mathbb Z$, $\mathbb R$ and $\mathbb C$ we shall denote the sets of integer, real and complex 
numbers respectively. The scalar product in a Hilbert space will be denoted by $\langle,\rangle$.

 We  need  following definitions. See \cite {K0}.

\begin{defn}
 A family of bounded  operators $\{H_\omega \}_{\omega \in \Omega} $ defined on Hilbert space $\H $ is {\it  weakly measurable}, if 
$ \Omega \ni \omega \longrightarrow \langle x, H_\omega y \rangle \in \mathbb C $ is measurable for every $x,y \in \H$. 
 A family of selfadjoint  operators $\{H_\omega \}_{\omega \in \Omega} $ is {\it  measurable }, if 
$\omega \longrightarrow (H_\omega-z)^{-1} $ is weakly measurable for all $z\in \mathbb C\backslash  \mathbb R$.
\end{defn}
 \begin{defn}
 A family of measurable transformations 
$T_i:\Omega \rightarrow \Omega$ , $i\in \mathbb Z^d$ is called  {\it measure preserving} if 
$\mathbb P(T_i^{-1}A)=\mathbb P(A)$ for every $A\in \mathcal M$ and  {\it ergodic} if  $T_i^{-1}A=A $  for all  $ i\in \mathbb Z^d $ implies
$\mathbb P(A)=0$ or $1$.
\end{defn}
\begin{defn}\label {ergop}
 Let $\{T_i\}_{i\in \mathbb Z^d}$ be measure preserving and ergodic.
A measurable family of selfadjoint operators $\{H_\omega \}_{\omega \in \Omega} $ on a separable Hilbert space $\mathcal H$
is called  $\mathbb Z^d$ {\it ergodic} if there exist a family $\{U_i\}_{i\in \mathbb Z^d}$ of unitary operators in $\H$ such that $H_{T_i\omega }=U_iH_\omega U^*_i$.
We call the family $\{H_\omega \}_{\omega \in \Omega} $  $\mathbb Z^d$ {\it stationary}, if we just require $\{T_i\}_{i\in \mathbb Z^d}$  to be  measure preserving,
(may be ergodic too).
\end{defn}

The following theorem was proven by L. Pastur \cite {P}. $\sigma (H_\omega)$ denotes the spectrum of $H_\omega$
\begin{thm} \label {Pastur}.  
Let $\{H_\omega \}_{\omega \in \Omega} $ be an ergodic family. 

 There exists $\Sigma \subset \mathbb R$  such that $\sigma (H_\omega)=\Sigma$ for $\mathbb P$
almost all  $\omega \in \Omega $, that is  for all $\omega \in \Omega _1 $ with $\mathbb P(\Omega _1)=1$.( $\Sigma $ is $\omega$ independent).
\end{thm}
 Analogous results hold for other parts of the spectrum. See \cite {KirM}, \cite {Kunz}
 
 \begin{thm}
Let $\{H_\omega \}_{\omega \in \Omega} $ be an ergodic family. There exist $\omega$ independent sets $\Sigma _{ac},\Sigma _{sc},\Sigma _{pp}$ such that 

$\Sigma _{ac} =\sigma _{ac}(H_{\omega})$ 
$\Sigma _{sc} =\sigma _{sc}(H_{\omega})$ 
$\Sigma _{pp} =\sigma _{pp}(H_{\omega})$ 
 for $ \mathbb P$ almost all $\omega $ .
\end{thm}

  The sets $\sigma _{ac}(H_{\omega}),\sigma _{sc}(H_{\omega}), \sigma _{pp}(H_{\omega}),$ denote the absolutely continuous, singular continuous and 
 pure point  spectra  of $H_{\omega}$ as defined in \cite {Teschl} p. 106, \cite{Last} or \cite {Kirsch} section 7.2 .The pure point spectrum $\sigma _{pp}(H_{\omega})$is the closure of the set of the eigenvalues of $H_{\omega}$.

   Finer decompositions of the spectra are possible. For ergodic operators in $l^2(\mathbb Z^d)$ the following result was proven in \cite{Last})
   (theorem 8.1),

 \begin{thm}\label{las}

For $\alpha \in [0,1]$ there exist subsets of $\mathbb R:\sigma _{\alpha ds},\sigma _{ed\alpha /\alpha s},\sigma _{\alpha ac},\sigma _{ed\alpha /s\alpha c}$ and $\sigma _{s\alpha dc}$ such that for $\mathbb P$  almost all $\omega$ they are respectively the $\alpha$-dimension singular , 
$\alpha$-singular of exact dimension $\alpha$ , absolutely continuous with respect to $h^\alpha$ , strongly $\alpha$-continuous  of exact dimension $\alpha$  and strongly $\alpha$-dimension continuous spectra of  $H_{\omega}$

\end{thm}
For the definition of all this kinds of different spectra , see \cite{Last}.\\

   If $H$ is an operator  with domain $D(H ) \subset \mathcal H$ in Hilbert space $\mathcal H$   and $M\subset \mathcal H$ is a subspace of
$\mathcal H$,    the {\it  restriction of  $H$ to $M$} denoted by $H| _{M} $  is the operator  with domain $D(H| _{M})= M  \cap D(H )$  and such that for $f\in D(H| _{M})$ one has  $H| _{M } f= Hf$. 

Let $P_M$ be the orthogonal projection on the closed subsapce $M \subset \mathcal H$ of the Hilbert space $\mathcal H$.
$M$ is said to {\it reduce} the symmetric operator $H$ or to be a {\it reducing subspace for $H$} if $u\in D(H)$ implies $P_Mu \in D(H)$
and $HP_Mu\in M$. See for example \cite {Ka} p. 278.

\begin{defn}\label{as}
   Assume  $M_\omega \subseteq  \H$ reduces the operator $H_\omega $.
A set $\Sigma \subset \mathbb R$ is called an {\it almost sure spectrum } for $\{H_\omega \}_{\omega \in \Omega} $ if there exist a 
set $\Omega _\Sigma \subset \Omega$ with  $\mathbb P( \Omega _\Sigma )=1$ such that 
$\Sigma =\sigma (H_\omega| _{M _\omega})$ for all $\omega \in  \Omega _\Sigma $, that is 
for $\mathbb P$ almost all  $\omega $. ($\Sigma$ is $\omega $ independent). If  $\mathbb P( \Omega _\Sigma )>0$ then we call $\Sigma $
a $\mathbb P$-{\it positive spectrum} for  $\{H_\omega \}_{\omega \in \Omega} $.

\end{defn}

\begin{remark}\label{as} From  Definition \ref{as} it follows that for $\mathbb P$ almost all $\omega \in \Omega $
$$ {\sigma (H_\omega), \sigma _{sc}(H_{\omega}), \sigma _{pp}(H_{\omega}),\sigma _{ac}(H_{\omega}), \sigma _{\alpha ds}(H_\omega), \atop\sigma _{ed\alpha /\alpha s}(H_\omega),\sigma _{\alpha ac}(H_\omega),
\sigma _{ed\alpha /s\alpha c}(H_\omega),\sigma _{s\alpha }(H_\omega)}$$
 
 are  almost sure spectra for $\{H_\omega \}_{\omega \in \Omega} $ if this family is ergodic.

 There can be almost sure spectra for operators which are not ergodic. See  for example \cite {LJ} Corollary 1.1.3.
\end{remark}




Now let us consider an $\mathbb Z^d$ stationary  family of selfadjoint operators $H_\omega $ acting on $L_2(\mathbb R^d)$ or $l_2(\mathbb Z^d)$.
 Let $\Lambda = [-1/2,1/2]^d\cap \mathbb R^d$ and denote by  $\chi _\Lambda  $  the characteristic function of
the set $\Lambda $, that is $\chi _\Lambda (x) =1$ if $x\in \Lambda $ and $\chi _\Lambda (x) =0$ if $x \not\in  \Lambda $.  

In case $H_\omega$ acts in $l_2(\mathbb Z^d)$ we talk of {\it the discrete model} and define 
   \begin{equation}
\nu (A):= \mathbb E( \langle\delta _0,E_{H_\omega}(A)\delta _0\rangle)
\end{equation}
 and in case $H_\omega $ acts in $L_2(\mathbb R^d)$ we talk of the {\it continuous model} and  define 
  \begin{equation} \label {dencont}
 \nu (A):= \mathbb E( tr \chi _\Lambda  E_{H_\omega}(A)\chi _\Lambda)   
\end{equation}
 for any Borel set $A$. 
$\mathbb E $ denotes the {\it mathematical expectation}, that is $\mathbb E (\dots)= \int_\Omega \dots d\mathbb P$ .  The symbol $\delta _0$
denotes the function such that   $\delta _0(0)=1$ and $\delta _0 (n) = 0$ for any $n\in \mathbb Z^d , n\not=0$. In general we define $\delta _i$ as $\delta _i(n)= 1$ if $i=n$ and $\delta _i (n)=0$ otherwise.
$E_{H_\omega}(A)$ is the {\it spectral projection measure} associated to  the selfadoint operator $H_\omega$, see \cite {Ka} p.355, that is $E_{H_\omega}(A)= \chi _A(H_\omega)$  where $\chi _A$ is 
the characteristic function of the Borel set $A \subset \mathbb R$. By $tr $ we denote the trace , which is uniquely defined (may be +$\infty$) for any bounded positive operator $B$ as $tr B= \sum_{n=1}^\infty \langle \varphi _n, B \varphi _n\rangle$ where $\{\varphi _n\}_{n=1}^\infty$ is an orthonormal basis. In equation (\ref{dencont}) $\chi _\Lambda$ is understood
 as a multiplication operator. 

The measure $\nu (A)$ defined above  is called {\it the density of states measure}.
 The distribution function $N$ of $\nu $   defined by 
 \begin{equation}
N(E)=\nu ((-\infty, E])
\end{equation}
is known as {\it the integrated density of states} IDS. See \cite {Kirsch}.
We shall use the short hand notation IDS  for the density of states measure or for the integrated density of states.
For more information on this object the interested reader can see \cite{VesM} and \cite{KM}.

Observe that the function $N$ is absolutely continuous if and only if the measure $\nu$ is absolutely continuous.

\section {Main results}\label{mai}

Let 
$\{H_\omega \}_{\omega \in \Omega} $ be a $\mathbb Z^d$ stationary family of selfadjoint operators acting on $L_2(\mathbb R^d)$ or $l_2(\mathbb Z^d)$
where the corresponding  unitary operators  are given by $(U_i\varphi)(x)  = \varphi (x-i) , i\in \mathbb Z^d$  for $\varphi \in 
L_2(\mathbb R^d)$ or $l_2(\mathbb Z^d)$. Let $I\subset \mathbb R$ be a closed interval which may be unbounded and denote its interior by $I^\circ$. 
  Our main theorem is 
 \begin{thm}\label{main}

If $\Sigma $ is a $\mathbb P$-positive spectrum for $\{H_\omega \}_{ \omega \in \Omega }$, then 
$\nu (\Sigma \cap I) =0$ implies $\Sigma\cap I^\circ  =\emptyset $, where $\nu$ is the density of states measure. 

\end{thm}

 For the proof we shall need Lemmas \ref{disc}, \ref{pro} and  \ref{proo} which are stated 
 later in this section.

\begin{proof}

A) {\sl Discrete model}.

{\it First step}.   
  
 From Lemma \ref {disc} we know that 
 $$ \nu (\Sigma ) =\mathbb E (\langle\delta _i,E_{H_\omega }(\Sigma\cap I )\delta _i \rangle)$$  for all $i\in \mathbb Z^d $.
 If we assume that $\nu (\Sigma\cap I )=0$,
 then for each fixed $ i\in \mathbb Z^d $   there exists a set $B_i \subset \Omega $ with $\mathbb P(B_i)=0$ such that
 for all $\omega \in \Omega \backslash B_i$ we have $ \langle\delta _i,E_{H_\omega }(\Sigma\cap I )\delta _i \rangle=0$. Consider the set 
 $\Omega':= \Omega \backslash(\cup _{i\in \mathbb Z^d}B_i)$, then $ \mathbb P(\Omega')=1 $ and for all $\omega \in \Omega'$,

 \begin{equation}\label {cero}
\langle\delta _i,E_{H_\omega }(\Sigma\cap I )\delta _i \rangle= \parallel E_{H_\omega }(\Sigma\cap I )\delta _i\parallel^2=0
\end{equation}
 Using  expression (\ref{cero}) we can see that for any $f\in l^2 (\mathbb Z^d)$ 
 
 \begin{equation}\label{f}
         \parallel E_{H_\omega }(\Sigma\cap I )f\parallel^2= \langle f,E_{H_\omega }(\Sigma\cap I )f \rangle=0 
\end{equation}
  for all $\omega \in \Omega'$. For this it is enough to write f in the basis $\{\delta _i\}_{i\in \mathbb Z^d}$, substitute in 
 (\ref{f}) and recall that $E_{H_\omega }(\Sigma\cap I )\delta _i$ for every $i$ is the zero vector by (\ref{cero}). Then 
$$ \parallel  E_{H_\omega }(\Sigma\cap I )f \parallel^2 = \parallel E_{H_\omega }(\Sigma\cap I )(\sum _i c_i \delta _i  )\parallel^2 =
\parallel \sum _i c_i E_{H_\omega }(\Sigma\cap I )\delta _i \parallel^2 =0
$$
     
 for all $\omega \in \Omega'$.
 
 Now, since $\Sigma$ is a $\mathbb P$- positive spectrum for $\{H_\omega \}_{\omega \in \Omega} $, (see Definition \ref{as}), there exists
 a set 
 $\Omega _\Sigma $ with $\mathbb P( \Omega _\Sigma )>0$ such that $\Sigma =\sigma (H_\omega|_{M _\omega})$ for all $\omega \in  \Omega _\Sigma $
 
 Set $\tilde\Omega  := \Omega'\cap \Omega _\Sigma $. Then $\mathbb P( \tilde\Omega )>0$ and therefore $\tilde\Omega\not=\emptyset$ and from (\ref{f}) we get
\begin{equation}\label{h}
\parallel E_{H_\omega }\bigl( \sigma (H_\omega| _{M _\omega})\cap I \bigr) f\parallel^2=0 
\end{equation}
for any  $f\in l^2 (\mathbb Z^d)$  if we require $\omega \in \tilde\Omega $. 

{\it Second step}.

 Fix $\omega_0 \in \tilde\Omega$ and
use the shorthand notation $S:=H_{\omega_0}| _{M _{\omega_0}} $.  By Lemma \ref{proo}, the set
 $M:= Rang E_S(I)$ is a reducing subspace for $S$.  The subspace $M \subset M _{\omega_0} $ is a reducing subspace for $H_{\omega_0}$ too.
One way to see this is to observe, using Lemma \ref{pro}, that the orthogonal projection $P_M$ onto the subspace $M$ is given by
$P_M= E_{H_{\omega _0}}(I) P_{M_{\omega _0}}$ where $P_{M_{\omega _0}}$ is the orthogonal projection onto $M_{\omega _0}$
and then notice that $P_M E_{H_{\omega _0}}(t)= E_{H_{\omega _0}}(t) P_M$ for all $t\in \mathbb R$. Commutation of the projection
with the spectral family is  known to 
be equivalent to reducibility, see Lemma \ref{proo} and \cite {Weid} thm. 7.28. 
 Now
\begin{equation}\label{hi}
                                       \parallel E_{H_{\omega_0}}\sigma (H_{\omega _0}|_M) f\parallel^2      = \parallel E_{H_{\omega_0}}\sigma (S|_M) f\parallel^2   \leq \parallel E_{H_{\omega_0}}(\sigma (S)\cap I) f\parallel^2 =0
\end{equation}
for every $f\in l_2(\mathbb Z^d)$. The first equality holds because from the  definition of restriction of an operator to a subspace we have 
$S|_M = \bigl (H_{\omega_0}| _{M _{\omega_0}}\bigr)|_M = H_{\omega _0}|_M$.
The inequality follows from Lemma \ref{proo} and the last equality from (\ref{h}).

If  $f\in M$,
$$E_{H_{\omega_0}}\bigl(\sigma (H_{\omega _0}|_M) \bigr)f=E_{H_{\omega_0}|_M}\bigl(\sigma (H_{\omega _0}|_M)\bigr) f =f$$
The first equality follows from  Lemma \ref{pro} and for the second, recall that 
$E_T(\sigma (T)) = id$, for any selfadjoint operator $T$,
 see for example Corollary 3.9 \cite {Teschl}. 
Using (\ref{hi}) then we conclude that $M=\{0\}$ and therefore
$$\sigma (H_{\omega_0}|_M) =\emptyset$$.
Hence
 $$\Sigma \cap I^\circ =\sigma (H_{\omega_0}|_{M_{\omega _0}}) \cap I^\circ =\sigma (S)\cap I^\circ  \subset \sigma (S|_M)= \sigma (H_{\omega_0}|_M) =\emptyset $$. 
 The contention follows from Lemma \ref{proo}.

B) {\sl Continuous model}.

{\it First step}

Assume $\nu (\Sigma \cap I )=0$. Then
 $$\nu (\Sigma \cap I )=\mathbb E( tr \chi _\Lambda  E_{H_\omega}(\Sigma \cap I )\chi _\Lambda)=\mathbb E( tr \chi _{\Lambda_i } E_{H_\omega}(\Sigma \cap I )\chi _{\Lambda_i})=0$$
  for every $i\in \mathbb Z^d$, according to  Lemma \ref{disc}, where $\Lambda _i := [-1/2+i,i+1/2]^d$.
  Let $\{\varphi _n\}_{n\in \mathbb N}$ be an orthonormal basis of $L_2(\mathbb R^d)$ and fix $i\in\mathbb Z^d$ .  Then 

\begin{eqnarray*}
\mathbb E( tr \chi _{\Lambda_i } E_{H_\omega}(\Sigma \cap I )\chi _{\Lambda_i})&=&
 \mathbb E\biggl( \sum _{n=1}^\infty \langle\varphi _n,\chi _{\Lambda_i } E_{H_\omega}(\Sigma \cap I )\chi _{\Lambda_i}\varphi _n\rangle \biggr)
= \\ &=&\mathbb E( \biggl (\sum _{n=1}^\infty     \parallel E_{H_\omega}(\Sigma \cap I )\chi _{\Lambda_i}\varphi _n\parallel^2 \biggr)=0  
\end{eqnarray*}
 and  therefore there exists a set $B_i\subset \Omega $  with $\mathbb P(B_i)=0$ such that
\begin{equation}\label{hh}
 \sum _{n=1}^\infty     \parallel E_{H_\omega}(\Sigma \cap I )\chi _{\Lambda_i}\varphi _n\parallel^2=0 
\end{equation}
for every $\omega \in \Omega \backslash B_i$ .
 Now take  $f\in L_2(\mathbb R^d)$ and write  $f$ 
with respect to the basis $\{\varphi _n\}_{n\in \mathbb N}$. We get, for any fixed $i\in  \mathbb Z^d$
\begin{equation}\label{ff}
\parallel E_{H_\omega}(\Sigma \cap I )\chi _{\Lambda_i}f\parallel =
 \parallel E_{H_\omega}(\Sigma \cap I )\chi _{\Lambda_i}(\sum _{n=1}^\infty  c_n \varphi _n )\parallel  
=\parallel \sum _{n=1}^\infty  c_n E_{H_\omega}(\Sigma \cap I )\chi _{\Lambda_i}\varphi _n \parallel =0
\end{equation}
for all $\omega \in \Omega \backslash B_i$.

The second equality follows from the continuity of the operators $E_{H_\omega}(\Sigma \cap I )$ and $\chi _{\Lambda_i}$
and the last equality because  $E_{H_\omega}(\Sigma \cap I )\chi _{\Lambda_i}\varphi _n$ is the zero vector for all $n$ almost surely,
which follows from (\ref{hh}).
 Observe that $f=\sum_{i\in \mathbb Z^d}\chi_{\Lambda_i}f $.
Then,
\begin{equation}\label{ojo}
\langle f,E_{H_\omega}(\Sigma \cap I )f\rangle=\langle f,E_{H_\omega}(\Sigma \cap I )\sum_{i\in \mathbb Z^d}\chi_{\Lambda_i}f\rangle =
\langle f,\sum_{i\in \mathbb Z^d}E_{H_\omega}(\Sigma \cap I )\chi_{\Lambda_i}f\rangle =0
\end{equation}
for every $ \omega \in \Omega':= \Omega \backslash(\cup _{i\in \mathbb Z^d}B_i)$.

{\it Second step}

 Follows as in the discrete case A).
\end{proof}
  As special case of theorem \ref{main} we have the following
  
\begin{corolary}\label{cor}
Let $\{H_\omega \}_{\omega \in \Omega} $ be as in theorem  \ref{main}. Assume moreover that this family is ergodic.
If the density of states measure $\nu (\cdot)$ is absolutely continuous with respect to a measure $\gamma (\cdot)$
then $\gamma (I\cap \sigma_\star)=0 $ implies  $I^\circ \cap \sigma_\star=\emptyset$, where $I$ is any closed interval and
$\sigma_\star$ is any almost sure spectrum for $H_\omega $. We could take for example $\star =sc, pp, ac$ or any of
the almost sure spectra mentioned in remark \ref{as}.

\end{corolary}

 The following lemmas are more or less standard facts.

\begin{lem}\label {disc}
  Let  $\{H_\omega \}_{\omega \in \Omega} $ be as theorem \ref{main}. Then, for every $i\in \mathbb Z^d$ \\

a)  $\nu (\cdot)=
\mathbb E( tr \chi _{\Lambda_i}  E_{H_\omega}(\cdot)\chi _{\Lambda_i}) $ in the continuous case\\
where $\Lambda _i := [-1/2+i,i+1/2]^d$. We shall write $\Lambda$  for    $\Lambda _0$.\\

b)   $\nu (\cdot)=\mathbb E( \langle\delta _i,E_{H_\omega}(\cdot)\delta _i\rangle)$ in the discrete case.\\

\end{lem}
\begin{proof}

First recall two facts:
If $h$ is a measurable function and $T_i$ is measure preserving, then 
\begin{equation}\label{li}
\mathbb E (h(T_i \omega ))= \int_{\Omega} h(T_i\omega )d \mathbb P(\omega )=
\int_{\Omega} h(\omega )d \mathbb P(T_i^{-1}\omega )=\int_{\Omega }h(\omega )d\mathbb P(\omega )=\mathbb E (h( \omega ))
\end{equation}.
 See for example \cite{Pet} p. 13.
  
The second is :
  
 If $S$ is a selfadjoint operator and $U$ unitary operator, then for any bounded measurable function $f$ we have
\begin{equation}\label{laa}
f(USU^*)=U f(S)U^* 
\end{equation}.
See Lemma 4.5 of \cite{Kirsch}.

Case a). Continuous model.

Let $h(\omega ):= tr \chi _\Lambda  E_{H_\omega}(A)\chi _\Lambda $.
  Then from (\ref{li}), for every $i \in \mathbb Z^d$,
\begin{equation}\label{lo}
\mathbb E( tr \chi _\Lambda  E_{H_\omega}(\cdot)\chi _\Lambda)  =\mathbb E( tr \chi _{\Lambda}  E_{H_{T_i\omega}}(\cdot)\chi _\Lambda)  
\end{equation}
 and from (\ref{laa}) 
\begin{equation}\label{lu}
\mathbb E( tr \chi _{\Lambda}  E_{H_{T_i\omega}}(\cdot)\chi _\Lambda)= \mathbb E( tr \chi _{\Lambda} U_i E_{H_{\omega}}(\cdot)U_i^*\chi _\Lambda) 
\end{equation}
 Let $\{\psi _n\}_{n\in \mathbb N}$ be an orthonormal basis for $L_2(\mathbb R^d)$ and denote $\varphi  _n:=U_i^*\psi  _n$.
\begin{eqnarray}\label{la}
\mathbb E( tr \chi _{\Lambda} U_i E_{H_{\omega}}(\cdot)U_i^*\chi _\Lambda) &=&
\mathbb E \biggl(\sum_n\langle  \psi _n,   \chi _{\Lambda} U_i E_{H_{\omega}}(\cdot)U_i^*\chi _\Lambda \psi _n \rangle \biggr )\nonumber\\
&=&\mathbb E\biggl(\sum_n\langle U_i^*\chi _{\Lambda}\psi _n,E_{H_{\omega}}(\cdot)U_i^*\chi _\Lambda \psi _n \rangle \biggr)\nonumber\\
&=&\mathbb E\biggl(\sum_n\langle \chi _{\Lambda}\varphi  _n,E_{H_{\omega}}(\cdot)\chi _\Lambda \varphi  _n \rangle \biggr)\nonumber\\
&=&\mathbb E\biggl (\sum_n\langle \chi _{\Lambda_i}\psi _n,E_{H_{\omega}}(\cdot)\chi _{\Lambda_i}\psi_n  \rangle \biggr)\nonumber\\
&=&\mathbb E( tr \chi _{\Lambda_i} E_{H_{\omega}}(\cdot)\chi _{\Lambda_i})
\end{eqnarray}
 Then from (\ref{lo})-(\ref{la}) , assertion a) of the Lemma follows.

Case b). Discrete model. 
\begin{eqnarray*}
\mathbb E( \langle\delta _0,E_{H_\omega}(\cdot)\delta _0\rangle)&=&\mathbb E(\langle \delta _0,E_{H_{T_i\omega}}(\cdot)\delta _0\rangle)\\
&=&\mathbb E(\langle \delta _0,U_i E_{H_{\omega}}(\cdot)U_i^*\delta _0\rangle)\\
&=&\mathbb E(\langle U_i^*\delta _0, E_{H_{\omega}}(\cdot)U_i^*\delta _0\rangle)\\
&=&\mathbb E(\langle \delta _{-i}, E_{H_{\omega}}(\cdot)\delta _{-i}\rangle)
\end{eqnarray*}
The first equality follows from (\ref{li}) and the second from (\ref{laa}).
 (In fact from \cite {LJ} and \cite{dRS} $\mathbb P$ almost surely the measures $\langle \delta _j,E_{H_\omega}(\cdot)\delta _j\rangle)$
 are equivalent for all $j$, in many cases.)
\end{proof}
\begin{lem}\label {pro}
Let $S$ be a selfadjoint operator in a Hilbert space $\mathcal H$ and $M$ a closed subspace of $\mathcal H$ 
which is a reducing subspace for $S$. Let $S_M:=S|_M$ be restriction of $S$ to $M$. Then $S_M$ is selfadjoint and for $t\in \mathbb R$
$$  E_S(t)|_ M  =E_{S_M}(t)$$ where $ E_S(t)$ and $E_{S_M}(t) $ denote the spectral families of orthogonal projections given by the spectral theorem   associated to the operators $S$ and $S_M$ respectively. $E_S(t)|_ M$ denotes the restriction of $E_S(t)$ to the subspace $M$.
\end{lem}
\begin{proof}
 The restriction of $E_S(t)$ to the subspace M, denoted by $E_S\bigl|_M(t)$, is a spectral family in the Hilbert Space $M$.





The properties required (see f.e. \cite {Weid} section 7.2):

i) $E_S| _M(t)^2=E_S| _M(t)$ and $E_S|_M(t) =E_S|_M(t)^*$, for every $t\in \mathbb R$

ii) $E_S|_M(s) \leq E_S|_M(t) $ for $s\leq t$ (monotonicity)

iii)$E_S|_M(t+\varepsilon ) \longrightarrow E_S|_M(t) $ for all $t\in \mathbb R$ as $\varepsilon  \longrightarrow 0+$ (continuity from the right)

iv)$E_S|_M(t) g \longrightarrow 0$ for every $g\in \mathcal H$  as $t\longrightarrow -\infty$, $E_S|_M(t) g\longrightarrow Id$ 
for every $g \in\mathcal H$  as $t\longrightarrow \infty$.

follow from the corresponding properties for $E_S(t)$.

According to (see \cite {Weid} thm. 7.28), the operator $S_M$ is selfadjoint. There exists therefore a 
corresponding spectral family $E_{S_M}(t)$ such that 
$S_M f  =\int\lambda d E_{S_M}(\lambda )f$ where  by definition the vector $\int\lambda d E_{S_M}(\lambda )f$
is the one such that  $\langle \int\lambda d E_{S_M}(\lambda )f, v\rangle = \int \lambda d \langle E_{S_M}(\lambda )f,v\rangle $ for $v\in M$
 see \cite {Ka}p. 356. Therefore we have for $f\in  D(S)\cap M$
 
 \begin{eqnarray*}
\langle \int\lambda d E_{S_M}(\lambda )f ,v \rangle &=& \langle S_M f ,v \rangle  =  \langle \int\lambda d E_{S}(\lambda )f ,v \rangle \\&=&
 \int \lambda d \langle E_S(\lambda )f ,v \rangle = \int \lambda d\langle E_S|_M f  , v\rangle \\ &=&
 \langle \int \lambda d(E_S(\lambda) |_M)f, v  \rangle 
\end{eqnarray*}

 Since the spectral family associated to a selfadjoint operator is unique, see \cite {Weid} thm 7.17, we obtain
 $$E_S(t )f =  E_S(t)|_ M f =E_{S_M}(t)f$$ if $f \in M$.

\end{proof}

Denote by $Rang T=\{Tf \mid f\in D(T)\}$ the range of the operator $T$.

In the following Lemma we use the same notation as in Lemma \ref{pro}.

\begin{lem}\label{proo}
It is possible to choose  $M=Rang E_S(I)$ as a reducing subspace for $S$ in Lemma \ref{pro}. Here  $I$ is a closed interval
$I= \{ x\in \mathbb R \mid a\leq x\leq b \}$, which could be unbounded ($a=-\infty$ or $b=\infty$ are allowed).
 In this case 
 $$\sigma (S)\cap I^\circ  \subset \sigma (S_M)  \subset \sigma (S)\cap I $$
where $I^\circ $ denotes the interior of $I$ 
\end{lem}
\begin{proof}
$Rang E_S(I)$ is a reducing subspace because $ E_S(I)E_S(t)=E_S(t)E_S(I) $ for all $t\in \mathbb R$ and a space is a reducing
subspace if and only if the projection on this space commutes with the spectral family of the operator. See \cite{Weid} theorem 7.28.

Recall that for a selfadjoint operator $T$,  $$\lambda \in \sigma (T) \quad\mbox {if and only if}\quad E_T((\lambda -\varepsilon ,\lambda +\varepsilon ))\not=0$$.
for every $\varepsilon >0$. See for example \cite{Teschl} thm 3.8. By Lemma \ref{pro}, 

$$ E_{S_M}((\lambda -\varepsilon ,\lambda +\varepsilon ))=
 E_{S}((\lambda -\varepsilon ,\lambda +\varepsilon ))|_M $$
 where $M= Rang E_S(I)$. Therefore to find the spectrum $\sigma (S_M)$ of $S_M$ we may look for the points 
 $\lambda$  for which there exists $g_\varepsilon \in M = Range E_S(I)$ such that $E_S(( \lambda -\varepsilon ,\lambda +\varepsilon ))g_\varepsilon \not=0$ 
 for every $\varepsilon >0$
Therefore $$\lambda \in \sigma (S_M) \Longleftrightarrow E_S((\lambda -\varepsilon ,\lambda +\varepsilon ))E_S(I)f_\varepsilon = 
E_S((\lambda -\varepsilon ,\lambda +\varepsilon )\cap I)f_\varepsilon \not=0$$
for some $f_\varepsilon \in \mathcal H$, for every $\varepsilon >0$, that is 
$$\lambda \in \sigma (S_M) \Longleftrightarrow E_S((\lambda -\varepsilon ,\lambda +\varepsilon )\cap I)\not=0$$
for every  $\varepsilon >0$.

Assume $\lambda \in I^\circ \cap \sigma (S)$. Take  $\varepsilon >0$ such that $(\lambda -\varepsilon ,\lambda +\varepsilon )\subset I$.
Since  $ \lambda \in \sigma (S)$ then  $$E_S((\lambda -\varepsilon ,\lambda +\varepsilon )\cap I)=E_S((\lambda -\varepsilon ,\lambda +\varepsilon ))\not=0$$
It can be seen this happens  for every $\varepsilon >0$ and then we obtain $$\sigma (S)\cap I^\circ  \subset \sigma (S_M) $$.

If $\lambda \not\in I$, then there exists $\varepsilon>0$ such that $(\lambda -\varepsilon ,\lambda +\varepsilon )\cap  I =\emptyset$. Then
$E_S((\lambda -\varepsilon ,\lambda +\varepsilon )\cap I)=E_S(\emptyset)=0$ and $\lambda \not\in \sigma (S_M)$. Therefore
$\sigma (S_M)\subset I$.
Since $E_S((\lambda -\varepsilon ,\lambda +\varepsilon )\cap I)\not=0$ implies $E_S((\lambda -\varepsilon ,\lambda +\varepsilon ))\not=0$,
then we have  $\sigma (S_M)\subset \sigma (S)\cap I$, and the Lemma is proved.

\end{proof}

\section {Examples}

 Here we shall consider situations where the IDS is absolutely continuous and therefore Theorem \ref{main} can be applied to
  obtain that  $|\Sigma \cap J|=0$ implies $\Sigma\cap J=\emptyset$, where $\Sigma$ is an almost sure spectrum and $J$ a closed interval. 
 We do not intend to be exhaustive and just mention some of the interesting cases.

 In \cite {H08} random  Schr\"{o}dinger operators of the form $H_{\omega }(\lambda )=H_0+\lambda V_\omega $ on $L_2 (\mathbb R^d)$ are considered for
 $\lambda \in \mathbb R$.
The  operator $H_0=(-i\nabla -A_0)^2 +V_0$ is nonrandom. The random Anderson type potential $V_\omega $ is 
constructed from the nonzero single-site potential $u\geq 0$ as
$$ V_\omega (x)=\sum _{j\in \mathbb Z^d}\omega_ju(x-j)$$
where the $\omega _j, j\epsilon \mathbb Z^d$ are random variables.

 Consider the hypotheses :
 
{\bf H1}) The background operator $H_0=(-i\nabla -A_0)^2 +V_0$ is lower semibounded, $\mathbb Z^d$-periodic
Schr\"{o}dinger operator with real valued,  $\mathbb Z^d$-periodic potential $V_0$ and a $\mathbb Z^d$-periodic
vector potential $A_0$. It is assumed that $V_0$ and $A_0$ are sufficiently regular  so that $H_0$ is essentially
selfadjoint on $C_0^\infty(\mathbb R^d)$.

{\bf H2}) The periodic operator $H_0$ has the unique continuation property, that is, for any $E \in \mathbb R$ and for any function
$\phi \in H_{loc}^2$, if $(H_0-E)\phi =0$ and if $\phi$ vanishes on an open set, then $\phi \equiv 0$

{\bf H3}) The nonzero, nonnegative, compactly-supported single-site potential $u\in  L_0^\infty(\mathbb R^d)$, with 
$\parallel u \parallel_{\infty}\leq 1$, and it is strictly positive on a nonempty open set.

{\bf H4} the random coupling constants $\omega _j, j\epsilon \mathbb Z^d$  are independent and 
identically distributed. The common distribution has density $ h_0 \in L_{\infty}(\mathbb R)$ with
supp $h_0 \subset [0,1]$.
 
Then according to  \cite{H08} Thm. 4.4 and \cite{chk07} Corollary 1.1). one has the following 
 \begin{thm}
Assume hypotheses H1)-H4). Then the IDS for the random family $H_\omega (\lambda )$, for $\lambda \not=0$ is locally Lipschitz 
continuous on $\mathbb R$.
\end{thm}
Therefore the IDS is absolutely continuous and we can apply Theorem \ref {main} as mentioned above.
 In \cite {H08} it is shown that under some  other hypothesis there is band-gap localization.
 This means that the spectrum close to the gaps is pure point almost surely. 
 If we are not near the gaps however, then it is  not clear  which kind of spectra we have and
 the fact that the IDS is absolutely continuous guarantees that there is not almost sure spectra 
 of zero measure in this region.
 
 The results just mentioned are quite restrictive about the single site potential $u$, since it is required to
 be nonnegative and of compact support.  There are results about the absolute continuity of the IDS where these conditions are relaxed.
 In \cite {Ves08} Thm. 1. for example,  single site potentials of  {\it generalized step 
function form} which are allowed to change sign are considered  and in \cite{CH94}
 Corollary 4.6. results  about the absolute continuity of the IDS are provided, where it is not required  the single site potential to be of compact support. 
 In \cite{HLMW} the authors prove the absolute continuity of the integrated density of states for multi-dimensional
Schr\"{o}dinger operators with constant magnetic field and ergodic random potential.  Examples of potentials to which these results apply
are certain  alloy type and  Gaussian random potentials. For Gaussian  potentials see \cite{VesG} too.
 
 For  the almost periodic case there are  results on the absolute continuity of the IDS too. 
The almost Mathieu operator is defined  on $l_2(\mathbb Z)$ by 
 $$   (Hu)_n =u_{n+1}+u_{n-1} +2 \lambda cos(2\pi[\theta  +n  \alpha ])u_n$$
  In \cite{AD} it is proven that the integrated density of states of $H$ is absolutely continuous 
 if and only if $|\lambda |\not=1$. 
  If $|\lambda |<1$, then the spectral measures of $H$ are absolutely continuous for almost every 
$\theta$ according to \cite{AD}.  It is known  the spectral measures have no absolutely continuous component for $|\lambda| \geq 1$.
  From our results follow  in particular that there is not almost sure singular continuous or pure point  spectrum of measure zero
 for these $\lambda $.

\begin{acknowledgments}
I thank Luis Silva for useful discussions and suggestions which improved this paper and Javier Rosenblueth for his help with Latex.
\end{acknowledgments}

\end{document}